\documentclass[12pt]{article}
\usepackage[a4paper]{geometry}
\usepackage{amsmath, amssymb, amsthm, blkarray}
\usepackage{relsize, multirow}
\usepackage{subcaption, float}
\usepackage{tikz}
\usepackage{tikz-cd}

\tikzstyle{vertex}=[circle, draw, inner sep=0pt, minimum size=6pt]

\newtheorem{thm}{Theorem}[section]

\newtheorem{lemma}[thm]{Lemma}
\newtheorem{prop}[thm]{Proposition}

\newtheorem*{thm*}{Theorem}
\newtheorem*{lemma*}{Lemma}
\newtheorem*{prop*}{Proposition}
\newtheorem*{cor*}{Corollary}
\newtheorem*{conj*}{Conjecture}

\theoremstyle{definition}
\newtheorem{defn}[thm]{Definition}

\newtheorem{ques}[thm]{Question}
\newtheorem{ex}[thm]{Example}

\newtheorem*{ques*}{Question}

\usepackage{graphicx} 

\title{Homomorphisms of Partial Fields}
\author{Nathaniel Vaduthala}
\date{\today}

\begin{document}

\maketitle

\begin{abstract}
    A partial field is an algebraic object that allows one to simultaneously abstract several different representability properties of matroids.
    In this paper we study partial fields as algebraic objects in their own right.
    We characterize the weak and strong characteristic sets of partial fields and show that the class of partial fields is not well-quasi ordered.
    We provide a new proof that the lift operator of a partial field is idempotent.
    We also provide a relation between the fundamental elements of a partial field and its Dowling lift, and show that the Dowling lift operator is idempotent. 
\end{abstract}

\section{Introduction}
Partial fields were first introduced by Semple and Whittle \cite{SEMPLE1996184} in order to provide a systematic generalization of various classes of matroids originating from matrix representations, such as regular matroids, representable matroids, or dyadic matroids. Pendavingh and van Zwam \cite{Pendavingh_2008, Pendavingh_2010} later built upon this theory of partial fields and in doing so were able to prove their Lift and Confinement Theorems, from which results such as Tutte's classification of regular matroids \cite{tutte_1965}, Whittle's classification of matroids representable over $GF(3)$ and other fields \cite{Whittle_1997}, and Vertigan's classification of golden ratio matroids \cite[Theorem 1.2.12]{vzThesis} directly follow. 

In this paper, we study partial fields as algebraic objects. 
Theorem~\ref{thm: characteristic sets} characterizes the sets that can arise as characteristic sets of partial fields,
and Theorem~\ref{thm: strong char sets} is the analogous result for strong characteristic sets.
The class of universal partial fields was conjectured to be well-quasi ordered in~\cite[Problem 3.4.6]{vzThesis}.
We do not resolve this conjecture, but we show that the class of \emph{all} partial fields is not well-quasi ordered in Theorem~\ref{thm: partial fields not wqo}. We provide a new proof that the lift operator of a partial field is idempotent in Proposition \ref{prop: lift is idempotent}. We also prove a relation between the fundamental elements of a partial field and the fundamental elements of its Dowling lift in Theorem \ref{thm: fundamental elem dowling}, and prove that the Dowling lift operator is idempotent in Theorem \ref{thm: dowling lift idempotent}.

\section{Background Information}\label{section: background}

We assume familiarity with ring theory and field theory. By a ring, we mean a commutative ring with unity.

\begin{defn}
    A \emph{partial field} $\mathbb{P} = (R, G)$ is a pair consisting of a commutative ring $R$ and a subgroup $G$ of the unit group $R^{\times}$ such that $-1 \in G$.
\end{defn}
We say that $p \in \mathbb{P}$, equivalently that $p$ is an \emph{element} of $\mathbb{P}$, if $p \in G$ or $p = 0$.

This definition of partial fields is equivalent to the one originally used by Semple and Whittle \cite{SEMPLE1996184}.

Although we will not be discussing any problems that directly involve representability of matroids over a partial field, we will still define it, as the notion will appear when providing motivation for ceertain problems.

\begin{defn}
    Given a partial field $\mathbb{P} = (R, G)$, a matrix $A$ is said to be a \emph{weak $\mathbb{P}$-matrix} if every entry of $A$ is in $R$ and every maximal minor of $A$ lies in $\mathbb{P}$. $A$ is said to be a \emph{strong $\mathbb{P}$-matrix} if every minor of $A$ lies in $\mathbb{P}$.
\end{defn}
Given a $r\times n$ weak $\mathbb{P}$-matrix $A$ with columns labelled by $[n]$, we can define a matroid $M(A)$ by its bases as follows
\begin{equation*}
    M(A) = ([n], \{X \subseteq [n] : |X| = r \text{ and } \det(A[X]) \neq 0\})
\end{equation*}
where $A[X]$ corresponds to the $r\times r$ submatrix of $A$ formed by restricting to the columns labelled by $X$. Every matroid that arises in this way is said to be \emph{weak $\mathbb{P}$-representable}. We can define a matroid over a strong $\mathbb{P}$-matrix in a similar way, and any matroid that arises from a strong $\mathbb{P}$-matrix is said to be \emph{strong $\mathbb{P}$-representable}. It is immediate that every matroid that is strong $\mathbb{P}$-representable is weak $\mathbb{P}$-representable, but the converse is true as well \cite[Proposition 2.3.3]{vzThesis}. Therefore, we will make no distinction of weak or strong $\mathbb{P}$-representability and instead simply say that a matroid is \emph{$\mathbb{P}$-representable}. 

Partial field rerepresentability for matroids generalizes several matroid repsentability concepts. For example, a matroid is representable over the partial field $(\mathbb{F}, \mathbb{F}^{\times})$ if and only if it is representable over the field $\mathbb{F}$; a matroid is representable over the partial field $\mathbb{U}_0 := (\mathbb{Q}, \{-1, 0, 1\})$ if and only if it is regular; and a matroid  representable over the partial field $\mathbb{D} := (\mathbb{Z}\left[ \frac{1}{2}\right], \langle -1, 2\rangle)$ if and only if it is dyadic.

One importance of the equivalence between weak and strong $\mathbb{P}=(R, G)$-representability is that it tells us that the restriction of the elements of $R$ to the elements of $\mathbb{P}$ is what directs matroid representability, and this serves as motivation for the definition of \emph{partial-field homomorphisms} and \emph{isomorphisms}, which we now state.
\begin{defn}
    Let $\mathbb{P}_1 = (R_1, G_1)$ and $\mathbb{P}_2 = (R_2, G_2)$ be partial fields. A function $\phi : \mathbb{P}_1 \to \mathbb{P}_2$ is a \emph{partial-field homomorphism} if 
    \begin{itemize}
        \item $\phi(1) = 1$
        \item for all $p, q \in \mathbb{P}_1$, $\phi(pq)=\phi(p)\phi(q)$
        \item for all $p, q, r \in \mathbb{P}_1$ such that $p+q=r$, $\phi(p)+\phi(q) = \phi(r)$
    \end{itemize}
    $\phi$ will be an \emph{isomorphism} if it is a homomorphism that satisfies the additional requirements
    \begin{itemize}
        \item $\phi$ is a bijection
        \item $\phi(p) + \phi(q) \in \mathbb{P}_2$ if and only if $p+q \in \mathbb{P}_1$
    \end{itemize}
\end{defn}

\begin{ex}
    Given a ring homomorphism $\phi : R \to R'$, the restriction $\phi|_{R^{\times}} : (R, R^{\times}) \to (R', (R')^{\times})$ is a partial-field homomorphism.
\end{ex}

One connection between partial field homomorphisms and matroid representability is that if a matroid $M$ is $\mathbb{P}$-representable, and if there exists a partial-field homomorphism $\phi : \mathbb{P} \to \mathbb{P}'$, then $M$ will be $\mathbb{P}'$-representable as well \cite[Corollary 2.9]{Pendavingh_2008}. It is because of this that we can immediately determine that every regular matroid is representable over every field \cite[Lemma 2.5.2]{Pendavingh_2008} and every dyadic matroid is representable over every field with characteristic that is not two \cite[Lemma 2.5.5]{vzThesis}.

Given partial fields $\mathbb{P}_1 = (R_1, G_1)$ and $\mathbb{P}_2 = (R_2, G_2)$, if $\phi : R_1 \to R_2$ is a ring homomorphism such that $\phi(G_1) \subseteq G_2$, then $\phi$ is a partial-field homomorphism $\phi : \mathbb{P}_1 \to \mathbb{P}_2$. This also holds if we replace \emph{homomorphism} will \emph{isomorphism}. Whenever a partial-field homomorphism (resp. isomorphism) arises in this way, we shall refer to it as a \emph{strong partial-field homomorphism} (resp. \emph{isomophism}). 
Not every partial-field homomorphism arises in this way.

\begin{ex}[{\cite[Example 2.2.5]{vzThesis}}]
    Let $R = \mathbb{F}_2 \times \mathbb{F}_3$, and let $\mathbb{P} := (R, R^{\times})$. Then we can define a partial-field homomorphism $\phi : \mathbb{P} \to \mathbb{U}_0$ by $\phi(0, 0) = 0$, $\phi(1, 1) = 1$, and $\phi(1, -1) = -1$. In particular, this is a partial-field isomorphism that cannot be extended to a ring homomorphism.
\end{ex}

As with any well-behaved algebraic object, compositions of partial field homomorphisms are again homomorphisms.

\begin{prop}\label{prop: composition of partial-field hom is hom}
    Let $\phi_1 : \mathbb{P}_1 \to \mathbb{P}_2$ and $\phi_2 : \mathbb{P}_2 \to \mathbb{P}_3$ be partial-field homomorphisms. Then $\phi_2 \circ \phi_1 : \mathbb{P}_1 \to \mathbb{P}_3$ is a partial-field homomorphism.
\end{prop}
\begin{proof}
    We have that $(\phi_2 \circ\phi_1)(1) = \phi_2(1) = 1$.

    For any $p, q \in \mathbb{P}_1$, we have 
    \begin{equation*}
        (\phi_2 \circ\phi_1)(pq) = \phi_2(\phi_1(p)\phi_2(q)) = \phi_2(\phi_1(p))\phi_2(\phi_1(q)) = (\phi_2 \circ\phi_1)(p) \cdot (\phi_2 \circ\phi_1)(q)
    \end{equation*}

    For any $p, q\in \mathbb{P}_1$ where $p+q\in \mathbb{P}_1$, we have
    \begin{align*}
        (\phi_2 \circ\phi_1)(p) + (\phi_2 \circ\phi_1)(q) &= \phi_2(\phi_1(p)) + \phi_2(\phi_1(q))\\
        &= \phi_2(\phi_1(p) + \phi_1(q))\\
        &= \phi_2(\phi_1(p+q))\\
        &= (\phi_2 \circ\phi_1)(p+q)\qedhere
    \end{align*}
\end{proof}

\section{Homomorphisms to Fields}\label{section: homomorphisms}

In this section, we discuss two problems related to homomorphisms from partial fields to fields and provide answers to them.

\subsection{Characteristic Sets}

The \emph{linear characteristic set} of a matroid $M$ is the set of all characteristics of fields in which $M$ is representable over. This can be adopted to the partial field setting. 

\begin{defn}
    Let $\mathcal{P} = \{0\} \cup \{p : p \text{ is prime}\}$ be the set of all primes, along with $0$. The \emph{(weak) characteristic set} of a partial field $\mathbb{P}$ is defined as
\begin{equation*}
       \chi(\mathbb{P}) =  \left\{ \, p \in \mathcal{P} \:\middle|\;
    \begin{array}{l}
        \text{there is a field  $\mathbb{F}$ of characteristic $p$ such that} \\
        \text{there is a partial-field homomorphism $\mathbb{P} \to \mathbb{F}$}
    \end{array}
    \right\}
\end{equation*}
\end{defn}

The following question about characteristic sets of partial fields was posed in~\cite{vzThesis}.
One of the main results of this section is an answer.

\begin{ques}[{\cite[Problem 2.8.5]{vzThesis}}]\label{ques: characteristic}
For which subsets $S \subseteq \mathcal{P}$ does there exist a partial field $\mathbb{P}$ such that $\chi(\mathbb{P}) = S$?
\end{ques}

Part of our answer to Question \ref{ques: characteristic} will involve a similar argument seen in \cite[Theorem 1]{LEADER1989315} that proves that if the linear characteristic set of a matroid contains infinitely many primes, then it contains $0$. This argument invokes \emph{ultrafilters} and \emph{ultraproducts} which are model-theoretic constructions.
We briefly provide the relevant information involving these structures now.
For more information see \cite[Chapter 5]{handbookanalysis}.

\begin{defn}
    Given a set $I$, a set $\mathcal{U} \subseteq 2^I$ is a \emph{non-principal ultrafilter over $I$} if it satisfies the following conditions
    \begin{itemize}
        \item if $A \in \mathcal{U}$ and $A \subseteq B$, then $B \in \mathcal{U}$
        \item if $A, B \in \mathcal{U}$ then $A \cap B \in \mathcal{U}$
        \item for all $A \subseteq I$, either $A \in \mathcal{U}$ or $I\setminus A \in \mathcal{U}$
        \item $\bigcap_{U \in \mathcal{U}}U = \varnothing$
    \end{itemize}
\end{defn}

\begin{defn}
    Let $(A_i)_{i \in I}$ be a collection of nonempty sets. Let $\mathcal{U}$ be an ultrafilter on the index set $I$. Given two elements $f = (f_i)_{i\in I}$ and $g = (g_i)_{i \in I}$ in the product $\prod_{i \in I}A_i$, we say that $f \sim_{\mathcal{U}} g$ if $\{i \in I : f_i = g_i\} \in \mathcal{U}$. The \emph{ultraproduct} is the set $\prod_{i \in I}A_i$ modulo the relation $f \sim_{\mathcal{U}} g$. We denote it by $\prod_{i \in I}A_i / \mathcal{U}$.
\end{defn}

Ultraproducts are relevant because of the following theorem.

\begin{thm}[{\cite[Theorems 2.1.5 and 2.4.1]{schoutens_ultraprod}}]\label{thm: ultraproduct of fields}
    Let $(\mathbb{F}_i)_{i \in I}$ be an infinite collection of fields and let $\mathbb{F} := \prod_{i \in I}\mathbb{F}_i/\mathcal{U}$ be an ultraproduct such that $\mathcal{U}$ is a non-principal ultrafilter on $I$. Then $\mathbb{F}$ is a field. Moreover, if for each prime $p$, if only finitely many $\mathbb{F}_i$ have characteristic $p$, then the characteristic of $\mathbb{F}$ is $0$.
\end{thm}

In order to use Theorem~\ref{thm: ultraproduct of fields},
we will need to know that non-principal ultrafilters indeed exist.
Zorn's Lemma allows one to construct an ultrafilter on any infinite set.

\begin{thm}[{\cite[Corollary 6.33]{handbookanalysis}}]\label{thm: ultrafilters exist}
    Assuming the \textbf{ZFC} set theory axioms, every infinite set has a non-principal ultrafilter.
\end{thm}

With this, we can now provide an answer to Question \ref{ques: characteristic}.

\begin{thm}\label{thm: characteristic sets}
    Let $S \subseteq \mathcal{P}$, then $S$ is the characteristic set of a partial field if and only if $S$ is nonempty and either $0 \in S$, or $0 \not\in S$ and $S$ is finite.
\end{thm}
\begin{proof}
\textbf{Case 1: $0 \in S$}

    We first consider when $0 \in S$. Let $A := \mathcal{P}\setminus S$, we define the partial field $\mathbb{P} := (R, R^{\times})$ where $R := \mathbb{Z}[1/q, q \in A]$. Then every element $r$ in $R$ can be written as 
    \begin{equation*}
        r = \frac{a}{s}, \quad a \in \mathbb{Z}, \quad s \in Q =\left\{ \prod_j q_j^{n_j} : q_j \in A, \, n_j \geq 0\right\}
    \end{equation*}
    Now, consider a $p \in S$ and consider a field $\mathbb{F}_p$ with characteristic $p$. We define a map $\phi_p : \mathbb{P} \to \mathbb{F}_p$ as follows
    \begin{equation*}
        \phi_p \left(\frac{a}{s}\right) = (a \mod(p))\cdot (s \mod(p))^{-1}
    \end{equation*}
    In the case of $p=0$, we define $x \mod(p) := x$. We first show that this map is well-defined, and then show it is a ring homomorphism, from which it follows it is a partial-field homomorphism. 

    Suppose $a/s = a'/s'$ in $R$, then there exists a $t \in Q$ such that $t(as' - a's) = 0$ in $\mathbb{Z}$. Reducing $\mod(p)$ gives us 
    \begin{equation*}
        (t\mod(p))\cdot ((a\mod(p))\cdot(s'\mod(p)) - (a'\mod(p))\cdot(s\mod(p)) = 0
    \end{equation*}
    Because $t \mod(p) \neq 0$, this implies
    \begin{equation*}
        (a\mod(p))\cdot(s'\mod(p)) - (a'\mod(p))\cdot(s\mod(p))=0
    \end{equation*}
    Rearranging this gives us
    \begin{equation*}
        \phi\left(\frac{a}{s} \right) = (a\mod(p))\cdot (s\mod(p))^{-1} = (a'\mod(p))\cdot (s'\mod(p))^{-1} = \phi\left(\frac{a'}{s'} \right)
    \end{equation*}
    Now we show it is a ring homomorphism. We first show additivity. Given $a/s, a'/s' \in R$, we have
    \begin{align*}
        \phi \left( \frac{a}{s} + \frac{a'}{s'}\right) &=\phi \left( \frac{as' + a's}{ss'}\right) \\
        &= ((as' + a's)\mod(p)) \cdot (ss' \mod(p))^{-1}
    \end{align*}
    using the fact that reduction $\mod(p)$ is a ring homomorphism, we can expand this term and simplify to obtain
    \begin{align*}
        ((as' + a's)\mod(p)) \cdot (ss' \mod(p))^{-1} &= (a\mod(p))\cdot (s\mod(p))^{-1}\\
        &+ (a'\mod(p))\cdot (s'\mod(p))^{-1} \\
        &= \phi \left( \frac{a}{s}\right) + \phi\left( \frac{a'}{s'}\right)
    \end{align*}
    Multiplicativity follows by a similar argument. 

    Now, if we consider an element $q \in A$ and then consider a field $\mathbb{F}_q$ of characteristic $q$, we can see there is no partial-field homomorphism $\psi : R \to \mathbb{F}_q$. If there was one, then we have $\psi(q) = q\cdot \psi(1_R) = 0$ in $\mathbb{F}_q$. But note that $q$ is a unit of $R$, and so $\psi(1_R) = \psi(q\cdot q^{-1}) = \psi(q)\cdot (q^{-1}) = 0$, which is a contradiction. 

\textbf{Case 2: $0 \not \in S$ and $S$ is nonempty and finite}

    We now consider the case when $0 \not\in S$ and $S$ is finite. We define $A = (\mathcal{P}\setminus\{0\})\setminus S$ as the nonzero values of $\mathcal{P}$ that are not in $S$. Consider the partial field $\mathbb{P} = (R, R^{\times})$, where $R' = \mathbb{Z}[1/q : q\in A]$, $m = \prod_{p \in S}p$, and $R = R'/(m)$. Similar to above, every element of $R'$ can be written as $a/s$, where $a \in \mathbb{Z}$ and $s$ can be written as a products of positive powers of the elements in $A$. Given $p \in S$, we let $\mathbb{F}_p$ be a field of characteristic $p$. We then define the map $\xi'_p : R' \to \mathbb{F}_p$
    \begin{equation*}
        \xi'_p\left(\frac{a}{s} \right) = (a\mod(p))\cdot(s\mod(p))^{-1}
    \end{equation*}
    We have that $(m) \subseteq \ker(\xi_p')$, and so there exists a unique induced ring homomorphism $\xi_p : R \to \mathbb{F}_p$. In particular, a residue class $[a/s]$ in $R$ gets mapped to
    \begin{equation*}
        \xi_p\left(\left[\frac{a}{s}\right] \right) = (a\mod(p))\cdot(s\mod(p))^{-1}
    \end{equation*}
    Now suppose there exists a partial field homomorphism $\xi : R \to \mathbb{F}_q$, where $\mathbb{F}_q$ is a field with characteristic $q \not\in S$. We then have $0_{\mathbb{F}_q} = \xi(0_R) = \xi(m) = m\cdot1_{\mathbb{F}_q}$. In order for this to be $0$, it must be that $q$ is a prime divisor of $m$, which implies then that $q \in S$. If $q=0$, then $m\cdot 1_{\mathbb{F}_q}$ can never be $0_{\mathbb{F}_q}$ then, which is a contradiction as well.

    \textbf{Case 3: $0 \not\in S$ and $S$ is infinite}

    We now suppose that $S$ is infinite and does not contain $0$, and we suppose there is a partial field $\mathbb{P}$ such that $\chi(\mathbb{P}) = S$. Consider the product $\prod_{p \in S}\mathbb{F}_p$, where $\mathbb{F}_p$ is a field of characteristic $p$ such that there is a partial-field homomorphism $f_p : \mathbb{P} \to \mathbb{F}_p$. We now define a map $\overline{f} : \mathbb{P} \to \prod_{p \in S}\mathbb{F}_p$ as $\overline{f} := (f_p)_{p \in S}$, which is a partial-field homomorphism by \cite[Lemma 2.18]{Pendavingh_2010}.
    
    By Theorem \ref{thm: ultrafilters exist}, there exists a non-principal ultrafilter $\mathcal{U}$ on $S$, and so the ultraproduct $\prod_{p \in S}\mathbb{F}_p / \mathcal{U}$ is a field by Theorem \ref{thm: ultraproduct of fields}. Moreover, the characteristic of this field will be $0$. We let $\pi :\prod_{p \in S}\mathbb{F}_p \to \prod_{p \in S}\mathbb{F}_p/\mathcal{U}$ be the canonical quotient map, and we consider the composition $\pi \circ \overline{f} : \mathbb{P} \to \prod_{p \in S}\mathbb{F}_p / \mathcal{U}$. Because this is the composition of partial-field homomorphisms, it is a partial-field homomorphism by Proposition \ref{prop: composition of partial-field hom is hom}, and so $0 \in S$, which is a contradiction.

    \textbf{Case 4: $S$ is empty}
    
    Assuming we are in \textbf{ZFC}, for every partial field $\mathbb{P}$, there exists a field $\mathbb{F}$ such that there is a homomorphism $\mathbb{P} \to \mathbb{F}$ \cite[Prop 2.2.6]{vzThesis}. Therefore, if $S \subseteq \mathcal{P}$ is the empty set, then there exists no partial field $\mathbb{P}$ such that $\chi(\mathbb{P}) = S$.
\end{proof}

A follow up question that can be asked is if the situation change if we restrict ourselves to only strong partial-field homomorphisms. As a reminder, $\phi : \mathbb{P} = (R, G) \to \mathbb{F}$ is a strong partial-field homomorphism if $\phi : R \to \mathbb{F}$ is a ring homomorphism and $\phi(G) \subseteq \mathbb{F}^{\times}$.

\begin{defn}
Given a partial field $\mathbb{P}$, the \emph{strong characteristic set} of $\mathbb{P}$ is defined as 
    \begin{equation*}
       \chi_{\text{strong}}(\mathbb{P}) =  \left\{ \, p \in \mathcal{P} \:\middle|\;
    \begin{array}{l}
        \text{there is a field  $\mathbb{F}$ of characteristic $p$ such that there} \\
        \text{is a strong partial-field homomorphism $\mathbb{P} \to \mathbb{F}$}
    \end{array}
    \right\}
\end{equation*}
\end{defn}

Theorem~\ref{thm: strong char sets} characterizes the subsets of $\mathcal{P}$ that can arise as strong characteristic sets:  they are exactly the sets that arise as (weak) characteristic sets of partial fields.
In order to prove it, we require Proposition~\ref{prop: ring implies partial} to reduce to a ring theoretic question, and a technical lemma about homomorphisms from product rings to fields.

\begin{prop} \label{prop: ring implies partial}
    Let $\mathbb{P} = (R, G)$ be a partial field and $\mathbb{F}$ be a field. Then $\phi: \mathbb{P} \to \mathbb{F}$ is a strong partial-field homomorphism if and only if $\phi : R\to \mathbb{F}$ is a ring homomorphism. 
\end{prop}
\begin{proof}
    If $\phi : \mathbb{P} \to \mathbb{R}$ is a strong partial-field homomorphism then by definition $\phi : R \to \mathbb{F}$ is a ring homomorphism. 

    If $\phi : R \to \mathbb{F}$ is a ring homomorphism, because ring homomorphisms map units to units, for any $g \in G$, we have that $\phi(g) \in \mathbb{F}^{\times}$. 
\end{proof}

\begin{lemma}\label{prop: product ring}
     Let $R = \prod_{i \in I}R_i$ be a product ring and let $\mathbb{F}$ be a field. Then there is a homomorphism $\phi : R \to \mathbb{F}$ if and only if there exists an $R_j$ such that there is a homomorphism $\psi : R_j \to \mathbb{F}$.
\end{lemma}
\begin{proof}
    If for some $j$, $R_j$ is homomorphic to $\mathbb{F}$ with $\psi$ as witness, then if $r \in R$ is written as $r = (r_1, r_2, \dots)$ we define $\phi : R \to \mathbb{F}$ as $\phi(r) = \psi(r_j)$. 

    Now suppose there is a homomorphism $\phi : R \to \mathbb{F}$. We let $e_i = (0, 0, \dots, 0, 1, 0, \dots, 0)$ where the nonzero term is in the $i$th location. We first show there exists a unique $e_j$ such that $\psi(e_j) \neq 0$. Suppose such an $e_j$ did not exist, so for every $i$, $\phi(e_i) = 0$. Then we have
    \begin{equation*}
        1_{\mathbb{F}} = \phi(1_R) = \phi\left(\sum e_i\right) = \sum\phi(e_i) = 0_{\mathbb{F}}
    \end{equation*}
    which is a contradiction, so there exists an $e_j$ that does not map to $0$. We now show such an $e_j$ is unique. If there existed another $e_k$ such that $\phi(e_k) \neq 0$ where $k \neq j$, then we have
    \begin{equation*}
        \phi(e_j)\phi(e_k) = \phi(e_j e_k) = 0_{\mathbb{F}}
    \end{equation*}
    which contradicts the fact that every field is an integral domain. 

    Now, given an element $r\in R$, we can write $r$ as $r = \sum r_i e_i$, which implies $\phi(r) = \sum r_i\phi(e_i)$. Because every $\phi(e_i) = 0$ for $i \neq j$, we have that $\phi(r) = r_j\phi(e_j) = \phi(0, 0, \dots, 0, r_j, 0, \dots)$. Therefore, the restricted map $\psi|_{R_j} : R_j \to \mathbb{F}$ is a homomorphism.
\end{proof}

We are now ready to give our characterization of strong characteristic sets.

\begin{thm}\label{thm: strong char sets}
    Let $S\subseteq \mathcal{P}$, then there exists a partial field $\mathbb{P}$ such that $\chi_{\text{strong}}(\mathbb{P}) = S$ if and only if $S$ is nonempty and either $0 \in S$ or $0 \not\in S$ and $S$ is finite. 
\end{thm}
\begin{proof}
    We will let $p$ denote an arbitrary element of $S$.
    
    Given a set $S$, we define $\mathbb{P} = (R, G)$ where $R = \prod_{p_i \in S} \mathbb{Z}_{p_i}$, where $\mathbb{Z}_{p_i}$ denotes the unique prime subfield of characteristic $p_i$, and $G = \prod_{p_i \in S}\mathbb{Z}_{p_i}^{\times}$ is the unit group of $R$. 

    \textbf{Case 1: $0 \in S$}

    Suppose that $\mathbb{F}$ has characteristic $p_j$ for some $p_j \in S$. Then $\mathbb{Z}_{p_j}$ will be a subfield of $\mathbb{F}$ and so there is a homomorphism from $\mathbb{Z}_{p_j}$ to $\mathbb{F}$, which implies there is a homomorphism from $R$ to $\mathbb{F}$ as well by Lemma \ref{prop: product ring}.

    Now suppose we have a field $\mathbb{F}$ that does not have characteristic $p_j$ for some $p_j \in S$. Then $R$ cannot be homomorphic to $\mathbb{F}$, because if it is, then there is some $\mathbb{Z}_{p_j}$ that is homomorphic to $\mathbb{F}$, despite having a different characteristic. 

    Therefore, the only fields that have a ring homomorpism from $R$ are fields with characteristic in $S$. Therefore, $\chi(\mathbb{P}) = S$.
    
    \textbf{Case 2: $0 \not\in S$ and $S$ is nonempty and finite}
    
    The argument in Case 1 works here as well.

    \textbf{Case 3: $0 \not\in S$ and $S$ is infinite}
    
    The same argument as shown in Case 3 in the proof of Theorem \ref{thm: characteristic sets} also shows that if any partial field $\mathbb{P}$ had a strong characteristic set equal to $S$, then we can construct a field $\mathbb{F}$ with characteristic $0$ such that there is a strong partial-field homomorphism $\mathbb{P} \to \mathbb{F}$.
    
     \textbf{Case 4: $S$ is empty} 

     Assuming we are in \textbf{ZFC}, given a partial field $\mathbb{P} = (R, G)$ there is a ring homomorphism from $R \to R/\mathfrak{m}$ where $\mathfrak{m}$ is a maximal ideal of $R$. Therefore, $\text{char}(R/\mathfrak{m}) \in \chi_{\text{strong}}(\mathbb{P})$ and so $S$ cannot be empty. 
\end{proof}

\subsection{Well-Quasi-Ordering of Partial Fields}

We now turn our attention to orderings of partial fields. In order to state the question involving orderings of partial fields, we first provide some preliminary definitions.

\begin{defn}
    Let $\succeq$ be a binary relation on a set $P$, then $\succeq$ is called a \emph{well-quasi-ordering} on $P$ if it satisfies the following conditions
    \begin{itemize}
        \item $a \succeq a$ for all $a \in P$
        \item if $a \succeq b$ and $b \succeq c$ then $a \succeq c$ for all $a, b, c \in P$
        \item for every infinite sequence of elements $a_1, a_2, a_3, \cdots$ from $P$, there exists a pair $a_i \succeq a_j$ where $i > j$
    \end{itemize}
\end{defn}

\begin{ex}
    The natural numbers under the standard ordering $(\mathbb{N}, \leq)$ are well-quasi-ordered.
\end{ex}

We can provide an ordering onto partial fields. 

\begin{defn}
    We say that $\mathbb{P}_2 \succeq_{\text{Hom}} \mathbb{P}_1$ if there is a partial-field homomorphism $\mathbb{P}_1 \to \mathbb{P}_2$.
\end{defn}

The motivation for this definition comes from the \emph{universal partial field} of a matroid. Given a matroid $M$ representable over at least one partial field, there exists a partial field $\mathbb{P}_M$ such that $M$ is representable over $\mathbb{P}_M$ and for every partial field $\mathbb{P}$ that $M$ is representable over, we have that $\mathbb{P} \succeq_{\text{Hom}} \mathbb{P}_M$. Such a partial field is referred to as the \emph{universal partial field of} $M$. Given a partial field $\mathbb{P}$, if there exists a matroid $M$ such that $\mathbb{P}$ is the universal partial field of $M$, then $\mathbb{P}$ is said to be a \emph{universal partial field}. We now state a question posed in \cite{vzThesis}.

\begin{ques}[{\cite[Problem 3.4.6]{vzThesis}}]
    Let $\mathbb{F}$ be a finite field. Under the relation $\succeq_{\text{Hom}}$, is the following set well-quasi-ordered?
    \begin{equation*}
        \{ \mathbb{P} : \mathbb{P} \text{ is universal and there is a partial-field homomorphism }\mathbb{P} \to \mathbb{F} \}.
    \end{equation*}
\end{ques}
The motivation for this comes from the fact that if Rota's conjecture is true, then this result immediately follows. A weaker variant of this question is also proposed in \cite{vzThesis}, where the universality requirement is dropped. It is remarked in \cite{vzThesis} that dropping this requirement may make the problem easier. However, if we drop this requirement, we can in fact construct an infinite descending chain.

\begin{thm}\label{thm: partial fields not wqo}
    Let $\mathbb{F}$ be a finite field. Under the relation $\succeq_{\text{Hom}}$ the following set is not well-quasi-ordered 
    \begin{equation*}
        \{ \mathbb{P} : \text{there is a partial-field homomorphism }\mathbb{P} \to \mathbb{F} \}
    \end{equation*}
\end{thm}
\begin{proof}
    Given a finite field $\mathbb{F}$, we define
    \begin{equation*}
        \mathbb{P}_i := \left( \mathbb{F}[x_1, x_2, \dots, x_i], \mathbb{F}^{\times} \right)
    \end{equation*}
    Let $a_1, a_2, \dots, a_i \in \mathbb{F}$, not all zero, be fixed elements. We see that there is a partial-field homomorphism $\mathbb{P}_i \to \mathbb{F}$ with the map $\psi_{a_1, a_2\dots, a_i} :\mathbb{F}[x_1, x_2, \dots, x_i] \to \mathbb{F}$, where $\psi_{a_1, a_2, \dots, a_i}(p(x_1, x_2, \dots, x_i)) = p(a_1, a_2,\dots, a_i)$, as witness.
    
    For any commutative ring $R$ that is an integral domain, $R[x]$ will be a commutative ring that is an integral domain as well. Furthermore, if $R$ is an integral domain, then the evaluation map $\phi_a : R[x]\to R$ given by $\phi_a(p(x)) = p(a)$ is a homomorphism.
    Therefore $\mathbb{F}[x_1, \dots, x_{i}]$ is a commutative ring that is an integral domain for every $i$. With this, because $\mathbb{F}[x_1, \dots, x_i] = \mathbb{F}[x_1, \dots, x_{i-1}][x_i]$, the evaluation map $\phi_a : \mathbb{F}[x_1, \dots, x_i] \to \mathbb{F}[x_1, \dots, x_{i-1}]$ where $\phi(p(x_1, \dots, x_i)) = p(x_1, \dots, x_{i-1}, a)$ for a fixed $a \in \mathbb{F}$, is a ring homomorphism. 
    
    In addition, we also have that $\phi_a(\mathbb{F}^{\times}) \subseteq \mathbb{F}^{\times}$, and so there is a partial-field homomorphism $\mathbb{P}_{i+1} \to \mathbb{P}_i$. Therefore, there is an infinite descending chain $\mathbb{P}_1 \succeq_{\text{Hom}} \mathbb{P}_{2} \succeq_{\text{Hom}} \mathbb{P}_3 \succeq_{\text{Hom}} \cdots$. 
\end{proof}

\section{Lifts of Partial Fields}

\subsection{Lift Ring}
If there is a partial-field homomorphism between two partial fields $\phi : \mathbb{P}_1 \to \mathbb{P}_2$, then any $\mathbb{P}_1$-representable matroid will be representable over $\mathbb{P}_2$ as well. The Lift Theorem, proven by Pendavingh and van Zwam \cite{Pendavingh_2010}, provides conditions for when representability over $\mathbb{P}_2$ implies representability over $\mathbb{P}_1$.
They also show that given a partial field $\mathbb{P}$, one can construct a partial field $\mathbb{LP}$ with homomorphism $\mathbb{LP}\rightarrow\mathbb{P}$ such that representability over $\mathbb{P}$ implies representability over $\mathbb{LP}$ by the Lift Theorem.
Moreover, in some sense, this partial field is the most general setting for which the Lift Theorem holds for a partial field. We now define this partial field $\mathbb{LP}$.

As a reminder, given a partial field $\mathbb{P} = (R, G)$ where $G$ is a subgroup of $R^{\times}$, we say $p \in \mathbb{P}$ if $p=0$ or $p \in G$.

\begin{defn}
    Given a partial field $\mathbb{P} = (R, G)$, we say that $p \in \mathbb{P}$ is a \emph{fundamental element} of $\mathbb{P}$, denoted by $p \in \mathcal{F}(\mathbb{P})$, if $1-p \in \mathbb{P}$.
\end{defn}

It follows from the definition of fundamental element that if $p \in \mathcal{F}(\mathbb{P})$, then $1-p \in \mathcal{F}(\mathbb{P})$ as well.

\begin{defn}
    Given a partial field $\mathbb{P} = (R, G)$, we define $\mathbb{X}_{\mathbb{P}} = \{X_p : p \in \mathcal{F}(\mathbb{P})\}$ to be a collection of indeterminates, one for each fundamental element of $\mathbb{P}$, $R_{\mathbb{P}} = \mathbb{Z}[\mathbb{X}_{\mathbb{P}}]$, and $I_{\mathbb{P}}$ is the ideal generated by the following polynomials in $R_{\mathbb{P}}$
    \begin{itemize}
        \item $X_0 - 0$; $X_1 - 1$ 
        \item $X_{-1} + 1$ if $-1 \in \mathcal{F}(\mathbb{P})$ 
        \item $X_p + X_q - 1$ where $p, q \in \mathcal{F}(\mathbb{P})$ and $p+q=1$
        \item $X_p X_q - 1$ where $p, q \in \mathcal{F}(\mathbb{P})$ and $pq=1$ 
        \item $X_p X_q X_r - 1$ where $p, q, r\in \mathcal{F}(\mathbb{P})$ and $pqr=1$
    \end{itemize}
    
    The \emph{lift} of $\mathbb{P}$, denoted as $\mathbb{L}\mathbb{P}$, is defined as follows
    \begin{equation*}
        \mathbb{L}\mathbb{P} = (R_{\mathbb{P}}/I_{\mathbb{P}}, \langle \{-1\} \cup \mathbb{X}_{\mathbb{P}}\rangle)
    \end{equation*}
\end{defn}

It was conjectured in \cite{Pendavingh_2010} that the lift operator of a partial field is idempotent, and this was later proven to be true in \cite[Corollary 2.17]{BAKER20251}. The machinery used to prove this in \cite{BAKER20251} involves the construction of a new lift of a pasture, referred to as the \emph{GRS-lift}, which can be viewed as the image of a canonical coreflection from the category of \emph{Pastures} onto a full subcategory $\mathcal{G}$ of \emph{Pastures}. From here it can be shown that there exists maps between the GRS-lift of a partial field and and the lift of a partial field that then allow idempotence of the lift operator to be established. 

We provide an alternative proof of idempotence of the lift operator that does not use an additional lift construction and instead provide a ring isomorphism between $\mathbb{L}^2\mathbb{P}$ and $\mathbb{LP}$ via the First Isomorphism Theorem. 

\begin{prop}[{\cite[Corollary 2.17]{BAKER20251}}]\label{prop: lift is idempotent}
    There is a strong partial field isomorphism between $\mathbb{L}^2\mathbb{P}$ and $\mathbb{L}\mathbb{P}$.
\end{prop}
\begin{proof}
    We first consider the rings $R_{\mathbb{P}} = \mathbb{Z}[\mathbb{X}_{\mathbb{P}}]$ and $R_{\mathbb{L}{\mathbb{P}}} = \mathbb{Z}[\mathbb{X}_{\mathbb{L}{\mathbb{P}}}]$. We can write $R_{\mathbb{P}}$ as $\mathbb{Z}[X_p : p \in \mathcal{F}(\mathbb{P})]$ and $R_{\mathbb{L}\mathbb{P}}$ as $\mathbb{Z}[Y_q : q \in \mathcal{F}(\mathbb{L}\mathbb{P})]$. We first show that every $X_p + I_{\mathbb{P}} \in \mathcal{F}(\mathbb{LP})$

    In this section, will denote $X_p + I_{\mathbb{P}}$ as $[X_{\mathbb{P}}]$. Consider an $[X_p] \in \langle \{-1\} \cup \mathbb{X}_{\mathbb{P}}\rangle$, so $p \in \mathcal{F}(\mathbb{P})$. We have that in $R_{\mathbb{P}}/I_{\mathbb{P}}$, $[X_{1-p}] = 1-[X_p]$. Because $1-p \in \mathcal{F}(\mathbb{P})$, we have that $1-[X_p] = [X_{1-p}] \in \ \langle \{-1\} \cup \mathbb{X}_{\mathbb{P}}\rangle$. Therefore, every $X_p$ is a fundamental element of $\mathbb{LP}$.
    
    With this, we can now see that the map $\phi : R_{\mathbb{L}{\mathbb{P}}} \to R_{\mathbb{P}}$ defined by $\phi(Y_q) = q$ where $q \in \mathcal{F}(\mathcal{P})$ for indeterminates and $\phi(k) = k$ for $k \in \mathbb{Z}$ is an isomorphism. Surjectivity follows by the fact that given an arbitrary $q$, there exists a $Y_{q}$ such that $\phi(Y_{q}) = q$, and for $k \in \mathbb{Z}$ we have $\phi(k) = k$. Injectivity immediately follows by definition of $\phi$. 

    Now consider the quotient map $\pi : R_{\mathbb{P}} \to R_{\mathbb{P}}/I_{\mathbb{P}}$, we define the map $\psi : R_{\mathbb{L}\mathbb{P}} \to R_{\mathbb{P}}/I_{\mathbb{P}}$ as $\psi = \pi \circ \phi$. We then have
    \begin{equation*}
        \ker(\psi) = \{r \in R_{\mathbb{L}\mathbb{P}} : \phi(r) \in \ker(\pi)\} = \phi^{-1}(\ker(\pi)) = \phi^{-1}(I_{\mathbb{P}})
    \end{equation*}
    However, we have that $\phi(I_{\mathbb{L}\mathbb{P}}) = I_{\mathbb{P}}$ which implies $\phi^{-1}(I_{\mathbb{P}}) = I_{\mathbb{L}\mathbb{P}}$ since $\phi$ is an isomorphism, and so $\ker(\psi) = I_{\mathbb{L}\mathbb{P}}$. Combining this with the fact that $\psi$ is the composition of surjective homomorphisms and therefore is a surjective homomorphism, by the First Isomorphism Theorem we have that $R_{\mathbb{L}\mathbb{P}}/I_{\mathbb{L}\mathbb{P}} \cong R_{\mathbb{P}}/I_{\mathbb{P}}$. In particular, the induced map $\hat{\psi} : R_{\mathbb{L}\mathbb{P}}/I_{\mathbb{L}\mathbb{P}} \to R_{\mathbb{P}}/I_{\mathbb{P}}$ where $\hat{\psi}(r + I_{\mathbb{LP}}) =\psi(r)+ I_{\mathbb{P}}$ is an isomorphism.

    We now show that $\hat{\psi}(\langle \{-1\} \cup \mathbb{X}_{\mathbb{LP}} \rangle) \subseteq \langle \{-1\} \cup \mathbb{X}_{\mathbb{P}} \rangle$.
    We have that $\hat{\psi}([Y_{q}]) = \psi(Y_{q}) = [q]$, where $q\in\mathcal{F}(\mathbb{L}\mathbb{P})$.
    So $q \in \{-1\} \cup \mathbb{X}_\mathbb{P}$.
\end{proof}

\subsection{The Dowling Lift of a Partial Field}

The lift of a partial field provides a general partial field in which the Lift Theorem is applicable. Another partial field that can be constructed to provide certain insights on representability, particularly representability of Dowling geometries, is the \emph{Dowling lift} $\mathbb{DP}$ of a partial field $\mathbb{P}$. In particular, consider the class of \emph{$\mathbb{P}$-graphic} matroids, where a matroid is $\mathbb{P}$-graphic if there is a $\mathbb{P}$-matrix $A$ where every column has at most two nonzero values such that $M = M(A)$. If a matroid is $\mathbb{P}$-graphic, then it is $\mathbb{DP}$-graphic as well.

\begin{defn}
    Let $\mathbb{P}=(R, G)$ be a partial field. We define $\overline{G}_{\mathbb{P}} := \langle \overline{G} \cup \overline{Q}_{\mathbb{P}}\rangle$, where $\overline{G}$ is an isomorphic copy of $G$ with elements $\{X_p : p \in G\}$, $\overline{Q}_{\mathbb{P}} = \{Y_p : p \in \mathcal{F}(\mathbb{P})\setminus\{0, 1\}\}$ is a set of indeterminates, $D_{\mathbb{P}}:= \mathbb{Z}[\overline{G}][\overline{Q}_{\mathbb{P}}]$ is the ring of polynomials in $\overline{Q}_{\mathbb{P}}$ over the group ring $\mathbb{Z}[\overline{G}]$, and $J_{\mathbb{P}}$ is the ideal in $D_{\mathbb{P}}$ generated by
    \begin{equation*}
        \{X_1 - 1\} \cup \{Y_p(1-X_p) - 1 : p\in \mathcal{F}(\mathbb{P})\setminus \{0, 1\}\}
    \end{equation*}
    
    The \emph{Dowling lift} of $\mathbb{P}$ is
    \begin{equation*}
        \mathbb{DP} = (D_{\mathbb{P}}/J_{\mathbb{P}}, \langle \{-1\} \cup \overline{G}_{\mathbb{P}}\rangle)
    \end{equation*}
\end{defn}

\begin{ques}[{\cite[Problem 3.4.1]{vzThesis}}]
    What is the relationship between $\mathcal{F}(\mathbb{P})$ and $\mathcal{F}(\mathbb{DP})$?
\end{ques}

\begin{thm}\label{thm: fundamental elem dowling}
    There is a bijective map between $\mathcal{F}(\mathbb{P})$ and $\mathcal{F}(\mathbb{DP}) \cap (\overline{G} \cup \{0\})$.
\end{thm}
\begin{proof}
    Consider the map $\phi : \mathcal{F}(\mathbb{P}) \to \mathcal{F}(\mathbb{DP}) \cap \overline{G}$ defined by $\phi(0) = 0$ and $\phi(p) = X_p$ for $p \in \mathcal{F}(\mathbb{P})\setminus \{0\}$. 

    To show injectivity, it suffices to consider the possibility that $X_p - X_q$ is in $\mathbb{J}_{\mathbb{P}}$ and show this cannot happen. If this were to happen, then $X_p - X_q$ could be written as the sum of polynomials of the form $f(X_1-1)$ and $g(Y_r(1-X_r)-1)$ where $r \in \mathcal{F}(\mathbb{P})\setminus\{0, 1\}$. We can see then that $X_p-X_q$ cannot be written in such a form. The fact that it cannot be written using $f(X_1-1)$ follows immediately, and the fact it cannot be written using $g(Y_r(1-X_r)-1)$ will follow from the fact that if it could, then there would necessarily need to exist a $Y_r$ term in $X_p-X_q$. From here, we obtain the fact that if $X_p = X_q$ then $p=q$.
    
    Now we need to show surjectivity. Let $X \in \mathcal{F}(\mathbb{DP}) \cap \overline{G}$, then $X \in \overline{G}$ and $1-X \in \{0\} \cup \langle \{-1\} \cup \overline{G}_{\mathbb{P}}\rangle$.
    

    Now, considering $X \in \mathcal{F}(\mathbb{DP}) \cap (\overline{G}\cup \{0\})$, if we have $X = 0$ then we have that $\phi(0) = X$. Now if we let $X \in \overline{G}$, then there is a unique $p \in G$ such that $X = X_p$. Now we show $p \in \mathcal{F}(\mathbb{P})$, and from there it follows $\phi(p) = X$. 

    As mentioned in \cite[Lemma 3.2.7]{vzThesis}, there is a partial-field homomorphism $\psi : \mathbb{DP} \to \mathbb{P}$. Because $X \in \mathcal{F}(\mathbb{DP})$, we have that $p = \psi(X) \in \mathcal{F}(\mathbb{P})$, because fundamental elements are mapped by partial-field homomorphisms to fundamental elements by \cite[Lemma 2.2.12]{vzThesis}. So therefore there exists $p \in \mathcal{F}(\mathbb{P})$ such that $\phi(p) = X$.
\end{proof}

We now prove that, similar to the lift of a partial field, the Dowling lift of a partial field is idempotent. To do so, we first prove a universal property of Dowling lifts.

As a reminder, if a partial-field homomorphism $\phi: \mathbb{P}_1 = (R_1, G_1)  \to \mathbb{P}_2 = (R_2, G_2)$ arises from a ring map $\phi : R_1 \to R_2$ such that $\phi(G_1) \subseteq G_2$, we call this a \emph{strong partial-field} homomorphism.

\begin{prop}\label{prop: universal property of dowling}
    Let $\mathbb{P} = (R, G)$ be a partial field, and let $\mathbb{P}' = (R', G')$ be a partial field such that there exists a strong partial-field homomorphism $\phi : \mathbb{P} \to \mathbb{P}'$. Then there exists a unique strong partial-field homomorphism $\psi : \mathbb{DP} \to \mathbb{P}'$ such that $\psi \circ i = \phi$, where $i : \mathbb{P} \to \mathbb{DP}$ is the canonical inclusion map.
\end{prop}
\begin{proof}
    Let $\phi : \mathbb{P} \to \mathbb{P}'$ be a strong partial-field homomorphism. We define a map $\Phi : \mathbb{Z}[\overline{G}][\overline{Q}_{\mathbb{P}}] \to \mathbb{P}'$. For $X_p \in \overline{G}$, we define $\Phi(X_p) = \phi(p)$. For $Y_p \in \overline{Q}_{\mathbb{P}}$, we define $\Phi(Y_p) = (1-\phi(p))^{-1}$. For $k \in \mathbb{Z}$ we define $\Phi(k) = k$.

    Now, we see that $J_{\mathbb{P}} \subseteq \ker(\phi)$, and so there exists a ring homomorphism $\psi : D_{\mathbb{P}}/J_{\mathbb{P}} \to \mathbb{P}'$. We have that $\psi(-1) = -1$, $\psi(X_p) = \phi(X_p)$, and $\psi(Y_p) = (1-\phi(p))^{-1}$ are all in $G'$, and so $\psi$ is a strong partial-field homomorphism. Furthermore, by construction we have that $\psi \circ i = \phi$, where $i : \mathbb{P} \to \mathbb{DP}$ is the inclusion map $i(p) = X_p$. 

    We now show uniqueness. Let $\psi' : \mathbb{DP} \to \mathbb{P}'$ be another strong partial-field homomorphism such that $\psi' \circ i = \phi$. Because of this relation, it has to be the case that $\psi'(X_p) = \phi(X_p)$ for $X_p \in \overline{G}$. Now, for any $Y_p \in \overline{Q}_{\mathbb{P}}$, in $D_{\mathbb{P}}/J_{\mathbb{P}}$ we have that $Y_p(1-X_p) =1$, and so $\psi'(Y_p)\cdot(1-\phi(X_p)) = 1$, and so $\psi'(Y_p) = (1-\phi(X_p))^{-1}$. And because there is only one ring homomorphisms from the integers to a ring, $\psi'(k) = k$ for $k \in \mathbb{Z}$. Therefore, $\psi' = \psi$.
\end{proof}

Equivalently, we can say that given a strong partial-field homomorphism $\phi : \mathbb{P} \to \mathbb{P}'$, there exists a unique strong partial-field homomorphism $\psi$ such that the following diagram commutes
\begin{equation*}
    \begin{tikzcd}
 \mathbb{P} \arrow[rd, "\phi"] \arrow[r, "i"] & \mathbb{DP} \arrow[d, "\psi"] \\
 & \mathbb{P}'
 \end{tikzcd}
\end{equation*}

\begin{thm}\label{thm: dowling lift idempotent}
    There is a strong partial field homomorphism between $\mathbb{D}^2\mathbb{P}$ and $\mathbb{DP}$.
\end{thm}
\begin{proof}
    By Proposition \ref{prop: universal property of dowling}, if we consider the identity map, which is a strong partial-field homomorphism, $\text{id}_{\mathbb{DP}} : \mathbb{DP} \to \mathbb{DP}$, there is a unique $\psi : \mathbb{D}^2\mathbb{P} \to \mathbb{DP}$ such that the following diagram commutes
    \begin{equation*}
    \begin{tikzcd}
 \mathbb{DP} \arrow[rd, "\text{id}_{\mathbb{DP}}"] \arrow[r, "i_{\mathbb{DP}}"] & \mathbb{D}^2\mathbb{P} \arrow[d, "\psi"] \\
 & \mathbb{DP}
 \end{tikzcd}
\end{equation*}
where $i_{\mathbb{DP}} : \mathbb{DP} \to \mathbb{D}^2\mathbb{P}$ is the inclusion map. We therefore have that $\text{id}_{\mathbb{DP}} = \psi \circ i_{\mathbb{DP}}$. It now suffices to show that $\text{id}_{\mathbb{D}^2\mathbb{P}} = i_{\mathbb{DP}}\circ \psi$.

Consider the following diagram
\begin{equation*}
    \begin{tikzcd}
 \mathbb{DP} \arrow[rd, "i_{\mathbb{DP}}"] \arrow[r, "i_{\mathbb{DP}}"] & \mathbb{D}^2\mathbb{P} \arrow[d, "h"] \\
 & \mathbb{D}^2\mathbb{P}
 \end{tikzcd}
\end{equation*}

    Consider the two maps $h_1 = \text{id}_{\mathbb{D}^2\mathbb{P}}$ and $h_2 = i_{\mathbb{DP}} \circ \psi$ where $i_{\mathbb{DP}}$ and $\psi$ are defined above. Because both of these would make the above diagram commute, by Proposition \ref{prop: universal property of dowling}, we have that $h_1=h_2$ and so we obtain the desired result. 
\end{proof}

\section{Acknowledgments}

Thanks to Daniel Irving Bernstein for the productive discussions involving partial fields and for insightful feedback. 

\bibliographystyle{abbrv}
\bibliography{refs}

\end{document}